\documentclass[oneside,english,reqno]{amsart}
\usepackage{geometry}
\usepackage{babel}
\usepackage{enumitem}
\usepackage{amssymb,amsmath,amsfonts,latexsym}
\usepackage{bbm}
\usepackage{appendix}
\usepackage{hyperref}
\usepackage{color}
\makeatletter
\numberwithin{equation}{section}
\numberwithin{figure}{section}
\theoremstyle{plain}
\newtheorem{theorem}{\protect\theoremname}
\theoremstyle{corollary}

\theoremstyle{remark}
\newtheorem{remark}[theorem]{\protect\remarkname}
\theoremstyle{plain}
\newtheorem{lemma}[theorem]{\protect\lemmaname}
\newlist{casenv}{enumerate}{4}
\setlist[casenv]{leftmargin=*,align=left,widest={iiii}}
\setlist[casenv,1]{label={{\itshape\ \casename} \arabic*.},ref=\arabic*}
\setlist[casenv,2]{label={{\itshape\ \casename} \roman*.},ref=\roman*}
\setlist[casenv,3]{label={{\itshape\ \casename\ \alph*.}},ref=\alph*}
\setlist[casenv,4]{label={{\itshape\ \casename} \arabic*.},ref=\arabic*}

\makeatother

\providecommand{\corollaryname}{Corollary}
\providecommand{\lemmaname}{Lemma}
\providecommand{\remarkname}{Remark}
\providecommand{\casename}{Case}
\providecommand{\theoremname}{Theorem}

\begin{document}

\title[Blow-up profile of neutron stars in the Chandrasekhar theory]{Blow-up profile of neutron stars \\ in the Chandrasekhar theory}

\author{Dinh-Thi Nguyen}
\address{Dinh-Thi Nguyen, Mathematisches Institut, Ludwig--Maximilians--Universit\"at M\"unchen, Theresienstr. 39, 80333 M\"unchen, Germany.} 
\email{nguyen@math.lmu.de}

	\subjclass{49J40}
	\keywords{Chandrasekhar limit, Chandrasekhar theory, gravitational interaction, Lane--Emden solution, mass concentration, minimizers, neutron stars, Thomas--Fermi theory}

	\maketitle

\begin{abstract}
We study the Chandrasekhar variational model for neutron stars, with or without an external potential. We prove the existence of minimizers when the attractive interaction strength $\tau$ is strictly smaller than the Chandrasekhar limit $\tau_c$ and investigate the blow-up phenomenon in the limit $\tau\uparrow \tau_{c}$. We show that the blow-up profile of the minimizer(s) is given by the Lane--Emden solution.
\end{abstract}

\section{Introduction}\label{intro}

It is a fundamental fact that a neutron star {\em collapses} when its mass is bigger than a critical number. The maximum mass of a stable star, called the \emph{Chandrasekhar limit}, was computed by Chandrasekhar in 1930 \cite{Ch-31}, which earned him the 1983 Nobel Prize in Physics.  In this paper, we study the details of the collapse phenomenon within the semi-classical approximation. 

From first principles of quantum mechanics, a neutron star is a system of identical, relativistic fermions interacting via the self-gravitational force. In the Chandrasekhar  theory, the ground state energy of a neutron star is given by
\begin{equation}\label{eq:neutron star energy}
E_{\tau}(1):=\inf\left\{ \mathcal{E}_{\tau}(\rho): 0\leq\rho\in L^{1}\cap L^{\frac{4}{3}}(\mathbb{R}^{3}),\int_{\mathbb{R}^3} \rho(x){\rm d}x=1\right\},
\end{equation}
with the energy functional
\begin{equation} \label{eq:TF-functional}
\mathcal{E}_{\tau}(\rho):=\int_{\mathbb{R}^3}j_{m}(\rho(x)){\rm d}x - \frac{\tau}{2}\iint_{\mathbb{R}^{3}\times\mathbb{R}^{3}}\frac{\rho(x)\rho(y)}{\left|x-y\right|}{\rm d}x{\rm d}y + \int_{\mathbb{R}^3}V(x)\rho(x){\rm d}x.
\end{equation}
Here $\rho$ is the density of the system and  $\tau>0$ stands for the interaction strength. The functional $j_{m}(\rho)$ is the semi-classical approximation for the  relativistic kinetic energy at density $\rho$, namely
\begin{align*}
	j_{m}(\rho)&=\frac{q}{(2\pi)^3}\int_{|p|<(6\pi^2\rho/q)^{\frac{1}{3}}}\sqrt{|p|^2+m^2}{\rm d}p\\
	& = \frac{q}{16\pi^2}\left[\eta(2\eta^2+m^2)\sqrt{\eta^2+m^2}-m^4\ln{\left(\frac{\eta+\sqrt{\eta^2+m^2}}{m}\right)}\right], \quad \eta=\left(\frac{6\pi^2\rho}{q} \right)^{\frac{1}{3}}.
\end{align*}
The mass $m> 0$ and the spin number $q\in \mathbb{N}$ will be fixed. Moreover, $V:\mathbb{R}^3\to \mathbb{R}$ stands for a general external potential; in the translation-invariant case $V\equiv0$ we will denote the corresponding energy functional and the ground state energy by $\mathcal{E}^{\infty}_{\tau}(\rho)$ and $E_{\tau}^{\infty}(1)$, respectively. 

The Chandrasekhar theory is the relativistic analogue of the famous Thomas--Fermi theory of non-relativistic electrons in atomic physics \cite{Fe-27,Th-27}. The rigorous derivation of the Chandrasekhar functional $\mathcal{E}_{\tau}(\rho)$ from many-body quantum theory has been done by Lieb and Yau in \cite{LiYa-87} (see \cite{LiTh-84} for an earlier work and \cite{FoLeSo-15} for a new approach). More precisely, they proved the validity of the Chandrasekhar theory from the $N$-body Schr\"odinger theory  in the limit of large $N$ with $\tau=gN^{\frac{2}{3}}$ kept fixed, where $g$ is Newton's gravitational constant. Their result holds under the condition that $\tau$ is strictly smaller than the Chandrasekhar limit $\tau_c$, which is described below. 

In the Chandrasekhar theory, the stellar collapse of big neutron stars boils down to the fact that  $E_{\tau}(1)=-\infty$ if $\tau>\tau_c$, where $\tau_c$ is the optimal constant in the inequality
$$
\int_{\mathbb{R}^3}j_{m}(\rho(x)){\rm d}x - \tau_c D(\rho,\rho) \ge 0, \quad \forall \, 0\le \rho\in L^1\cap L^{\frac{4}{3}}(\mathbb{R}^3).
$$
Here we have introduced the direct energy term 
$$
D(\rho,\rho) = \frac{1}{2}\iint_{\mathbb{R}^{3}\times\mathbb{R}^{3}}\frac{\rho(x)\rho(y)}{\left|x-y\right|}{\rm d}x{\rm d}y.
$$

From the operator inequality $|p| \leq \sqrt{|p|^2+m^2} \le |p|+m$ and a standard scaling argument, we can see that $\tau_{c}$ is independent of $m$. Since
$$
\lim_{m\to 0} j_m(\rho)= K_{\rm cl} \rho^{\frac{4}{3}}, \quad K_{\rm cl}:=\frac{3}{4}\left(\frac{6\pi^2}{q}\right)^{\frac{1}{3}}
$$
we find that 
$$\tau_{c}=\sigma_{f}^{-1}K_{\rm cl}$$ 
where $\sigma_{f}$ is the optimal constant in the inequality
\begin{equation}\label{ineq:HLS}
\sigma_{f}\|\rho\|_{L^{\frac{4}{3}}}^{\frac{4}{3}}\|\rho\|_{L^{1}}^{\frac{2}{3}} \geq D(\rho,\rho),\quad \forall 0\leq\rho\in L^{1}\cap L^{\frac{4}{3}}(\mathbb{R}^{3}).
\end{equation}
Numerically, $\sigma_{f}\approx 1.092$ (we use the notation in  \cite{LiYa-87} where $f$ stands for fermions). 

It is well-known (see \cite[Appendix A]{LiOx-81}) that \eqref{ineq:HLS} has an optimizer $Q$, which is unique up to dilations and translations. Moreover, such $Q$ can be chosen uniquely to be non-negative, symmetric decreasing, and satisfies
\begin{equation}\label{cond:LE}
\sigma_{f}\|Q\|_{L^{\frac{4}{3}}}^{\frac{4}{3}} = \|Q\|_{L^{1}}^{\frac{2}{3}} = D(Q,Q) = 1.
\end{equation}
This function $Q$ solves the Lane--Emden equation of order $3$ (see \cite{La-70,Th-27,GaBaHa-80})
\begin{equation}\label{eq:eq:LE}
\frac{4}{3}\sigma_{f} Q(x)^{\frac{1}{3}}-(|\cdot|^{-1}\star Q)(x)+\frac{2}{3} \quad \begin{cases}
=0 & \text{if } Q(x)>0,\\
\geq0 & \text{if } Q(x)=0.
\end{cases}
\end{equation}

In the present paper, we analyze the existence and blow-up behavior of the minimizers of the variational problem $E_{\tau}(1)$ in \eqref{eq:neutron star energy} when $\tau$ approaches $\tau_{c}$ from below. 

Our first result is
\begin{theorem}[Existence of minimizers]
	\label{thm:existence of minimizers}
	Fix $q\geq 1$ and $m>0$. Assume that $V$ satisfies 
\begin{itemize}
	\item[($V_1$)] $0 > V\in L^{4}(\mathbb{R}^3)+L^{\infty}(\mathbb{R}^3)$, and \item[($V_2$)] $V$ vanishes at infinity, i.e. $|\{x:|V(x)|>a\}|<\infty$ for all $a>0$.
\end{itemize}
Then the variational problem $E_{\tau}(1)$ in \eqref{eq:neutron star energy} has the following properties
	\begin{itemize}
		\item[(i)] If $\tau>\tau_{c}$, then $E_{\tau}(1)=-\infty$;
		\item[(ii)] If $\tau=\tau_{c}$, then $E_{\tau}(1)=\inf_{x\in\mathbb{R}^3}V(x)$ but it has no minimizer;
		\item[(iii)] If $0< \tau<\tau_{c}$, then $E_{\tau}(1)$ has at least one minimizer.
	\end{itemize}
\end{theorem}

Here we focus on the case where $V$ is attractive and vanishes at infinity. We assume $V\in L^{4}(\mathbb{R}^3)+L^{\infty}(\mathbb{R}^3)$ in order to ensure that the term $\int_{\mathbb R^3} V(x) \rho(x){\rm d}x$ is meaningful when $\rho\in L^{1}\cap L^{\frac{4}{3}}(\mathbb{R}^{3})$. The value $\inf_{x\in\mathbb{R}^3}V(x)$ in (ii) should be interpreted properly as the essential infimum when $V$ is a general, measurable function. 

In the case $V \equiv  0$, Theorem \ref{thm:existence of minimizers} is well-known (see \cite{AuBe-71,LiYa-87}). Moreover, when $0< \tau<\tau_{c}$, the minimizer is unique up to translations and can be chosen to be radially symmetric decreasing by the rearrangement inequalities (see \cite[Theorem 5]{LiYa-87}). For $V \not \equiv 0$, the existence result in Theorem \ref{thm:existence of minimizers} is non-trivial and requires a concentration-compactness argument \cite{Li-84} in order to deal with the lack of compactness of minimizing sequences.

Our next result concerns the behavior of the minimizers of $E_{\tau}(1)$ as $\tau\uparrow \tau_c$. We will show that the minimizers of $E_{\tau}(1)$ blows up  and that its blow-up profile is given by the unique optimizer of \eqref{ineq:HLS}. To make the analysis precise, let us assume that the external potential $V$ is either $0$ or of the typical form (relevant for physics)
\begin{equation}\label{external potential}
V(x)=-\sum_{i=1}^{M}\frac{z_{i}}{|x-x_i|^{s_{i}}},
\end{equation} 
where $0<z_{i}$, $0<s_{i}<\frac{3}{4}$, $x_i\in\mathbb{R}^3$ and $x_i\ne x_j$, for $1\leq i\ne j\leq M$. Let 
$$
s=\max\{s_{i}:1\leq i\leq M\}, \quad z=\max \{z_{i}:s_{i}=s\}, \quad \mathcal{Z}=\{x_i:s_{i}=s \text{ and } z_{i}=z\}.
$$
Thus $\mathcal{Z}$ denotes the locations of the most singular points of $V(x)$. We have

\begin{theorem}[Blow-up of minimizers]
	\label{thm:behavior} Fix $q\geq 1$ and $m>0$. For $0< \tau<\tau_{c}$, let $\rho_{\tau}$ be a minimizer of $E_{\tau}(1)$. Then for every sequence $\{\tau_{n}\}$ with $\tau_{n}\uparrow \tau_{c}$ as $n\to\infty$, the following hold true. 
	\begin{itemize}
		\item[(i)] If $V\equiv 0$, then there exist a subsequence of $\{\tau_{n}\}$ (still denoted by $\{\tau_{n}\}$) and a sequence $\{x_n\}\subset \mathbb{R}^3$ such that  
		\begin{equation}\label{conv:minimizer-infinity}
		\lim_{n\to\infty}(\tau_{c}-\tau_{n})^{\frac{3}{2}}\rho_{\tau_{n}}((\tau_{c}-\tau_{n})^{\frac{1}{2}}x+x_n)=\lambda_{\infty}^{3}Q\left(\lambda_{\infty} x\right)
		\end{equation}
		strongly in $L^{1}\cap L^{\frac{4}{3}}(\mathbb{R}^{3})$. 		Here $Q$ is the unique non-negative radial function satisfying \eqref{cond:LE}--\eqref{eq:eq:LE} and
		$$
		\lambda_{\infty} = \frac{3}{4}m\left(\frac{1}{K_{\rm cl}}\int_{\mathbb{R}^3}Q(x)^{\frac{2}{3}}{\rm d}x\right)^{\frac{1}{2}}.
		$$
		\item[(ii)] If $V$ is defined as in \eqref{external potential}, then there exist a subsequence of $\{\tau_{n}\}$ (still denoted by $\{\tau_{n}\}$) and an $x_j\in\mathcal{Z}$ such that
		\begin{equation}\label{conv:minimizer-V}
		\lim_{n\to\infty}(\tau_{c}-\tau_{n})^{\frac{3}{1-s}}\rho_{\tau_{n}}((\tau_{c}-\tau_{n})^{\frac{1}{1-s}}x+x_j)=\lambda_s^{3}Q\left(\lambda_s x\right)
		\end{equation}
		strongly in $L^{1}\cap L^{\frac{4}{3}}(\mathbb{R}^{3})$, with 
		$$
	\lambda_{s} = \left(sz\int_{\mathbb{R}^3}\frac{Q(x)}{|x|^{s}}{\rm d}x\right)^{\frac{1}{1-s}}.
	$$ 
		If $\mathcal{Z}$ has a unique element, then \eqref{conv:minimizer-V} holds for the whole sequence of $\tau_n$. 
	\end{itemize}
\end{theorem}

Note that when $V\equiv 0$ or $V$ is defined by \eqref{external potential}, the existence of minimizer when $0< \tau<\tau_{c}$ has been proved in Theorem \ref{thm:existence of minimizers}. Our proof of Theorem \ref{thm:behavior} is based on a detailed analysis of the Euler--Lagrange equation associated to the minimizers of $E_{\tau}(1)$ when $\tau$ tends to $\tau_{c}$. As a by-product of our proof, we also obtain the asymptotic behavior of the energy, that is 
\begin{equation} \label{lim:e(tau)/epsilon^0}
\lim_{\tau\uparrow \tau_{c}} \frac{E_{\tau}^{\infty}(1)}{(\tau_c-\tau)^{\frac{1}{2}}} = \frac{3}{2}m\left(\frac{1}{K_{\rm cl}}\int_{\mathbb{R}^3}Q(x)^{\frac{2}{3}}{\rm d}x\right)^{\frac{1}{2}}
\end{equation}
if $V\equiv 0$, and 
\begin{equation} \label{lim:e(tau)/epsilon^s}
\lim_{\tau\uparrow \tau_{c}} \frac{E_{\tau}(1)}{(\tau_c-\tau)^{\frac{s}{s-1}}} = \left(1-\frac{1}{s}\right) \left(sz\int_{\mathbb{R}^3}\frac{Q(x)}{|x|^{s}}{\rm d}x\right)^{\frac{1}{1-s}}
\end{equation}
if $V$ is defined as in \eqref{external potential}.

In the case of boson stars, the ground state energy can be approximately captured  by the pseudo-relativistic Hartree-type functional \cite{LiYa-87}. In this case, the blow-up analysis has been carried out recently in  \cite{GuZe-17,YaYa-17,Ng-17-boson} (see also \cite{Ng-19-boson}). The method in these works is inspired by Guo and Seiringer \cite{GuSe-13} who studied the mass concentration of the Bose--Einstein condensate described by the 2D focusing Gross--Pitaevskii equation (see also \cite{LeNaRo-17} for an extension to the rotating case). 

The Chandrasekhar model studied in the present paper is a semi-classical theory and the Lane--Emden equation \eqref{eq:eq:LE} is different from the Hartree-type equations in \cite{GuSe-13,LeNaRo-17,GuZe-17,YaYa-17,Ng-17-boson,Ng-19-boson}. This requires new ideas in order to prove both existence and blow-up results. We hope that our study can serve as a first step to understand the blow-up phenomenon of neutron stars in a rigorous mathematical approach. 

\medskip

We will prove Theorem \ref{thm:existence of minimizers} in Section \ref{sec:existence-minimizer}, and Theorem \ref{thm:behavior} in Section \ref{sec:behavior}. 

\section{Existence of minimizers} \label{sec:existence-minimizer}

In this section, we prove the existence and non-existence of minimizers of $E_{\tau}(1)$ as stated in Theorem \ref{thm:existence of minimizers}. The exsistence and non-exsitence of minimizer of $E_{\tau}^{\infty}(1)$ when $V\equiv 0$ is well-known result (see \cite{AuBe-71,LiYa-87}). Here we consider the case $V \not \equiv 0$ which satisfies conditions $(V_1)$--$(V_2)$.

Let $0< \tau < \tau_c$, and let $\{\rho_{n}\}$ be a minimizing sequence of $E_{\tau}(1)$, i.e.
$$
\lim_{n\to \infty}\mathcal{E}_{\tau}(\rho_{n})=E_{\tau}(1), \text{ with } \rho_{n} \in L^{1}\cap L^{\frac{4}{3}}(\mathbb{R}^{3}) \text{ and }\int_{\mathbb{R}^3}\rho_{n}(x){\rm d}x=1, \text{ for all } n.
$$
First of all, we see that all the terms of the energy functional $\mathcal{E}_{\tau}$ in \eqref{eq:TF-functional} are well-defined. Indeed, it follows from the inequality $\sqrt{|p|^{2}+m^{2}}\leq|p|+m$ that $j_{m}(\rho_{n})\leq K_{\rm cl}\rho_{n}^{\frac{4}{3}}+m\rho_{n}$, which shows that the kinetic energy is well-defined. On the other hand, since $\rho_{n}\in L^{1}\cap L^{\frac{4}{3}}(\mathbb{R}^{3})$
we have $\rho_{n}\in L^{r}(\mathbb{R}^{3})$ for any $1\leq r\leq \frac{4}{3}$ by interpolation, then it follows from the Hardy--Littlewood--Sobolev inequality (see \cite[Theorem 4.3]{LiLo}) that the direct term $D\left(\rho_{n},\rho_{n}\right)$ is well-defined. Finally, the conditions $(V_1)$--$(V_2)$ imply that the external potential term $\int_{\mathbb{R}^3}V(x)\rho_{n}(x){\rm d}x$ is well-defined. 

Next, we collect some basic facts.


\begin{lemma}[Binding inequality]\label{lem:ineq:binding}
	Fix $q\geq 1$ and $m>0$. Assume that $V \not \equiv 0$ satisfies conditions $(V_1)$--$(V_2)$. Then for any $0<\alpha<1$ we have
	\begin{equation}\label{ineq:binding-0}
	E_{\tau}(1)\leq E_{\tau}(\alpha)+E_{\tau}^{\infty}(1-\alpha).
	\end{equation}
\end{lemma}

\begin{proof}
	Assume, by contradiction, that there exists $\delta>0$ and $0<\alpha<1$ such that
	\begin{equation}\label{ineq:binding-1}
	E_{\tau}(1)>E_{\tau}(\alpha)+E_{\tau}^{\infty}(1-\alpha)+\delta.
	\end{equation}
Then there exist states $\rho_{0}$ and $\rho_{\infty}$ in $L^{1}\cap L^{\frac{4}{3}}(\mathbb R^3)$ with $\int_{\mathbb{R}^{3}}\rho_{0}(x){\rm d}x = \alpha = 1 - \int_{\mathbb{R}^{3}}\rho_{\infty}(x){\rm d}x$ such that $E_{\tau}(\alpha)>\mathcal{E}_{\tau}\left(\rho_{0}\right)-\delta/3$ and $E_{\tau}^{\infty}(1-\alpha)>\mathcal{E}^{\infty}_{\tau}\left(\rho_{\infty}\right)-\delta/3$. By a density argument, we can assume that $\rho_{0}$ and $\rho_{\infty}$ are compactly supported. Denote by $R>0$ the radius of a ball in $\mathbb{R}^{3}$ which contains the supports of $\rho_{0}$ and $\rho_{\infty}$. We define a translated operator by 
	$$
	\widetilde{\rho}_{\infty}(x):=\rho_{\infty}(x-3R)
	$$
	and a trial density operator $\rho_{\alpha}:=\rho_{0}+\widetilde{\rho}_{\infty}$.	By construction we have $\rho_{0}\tilde{\rho}_{\infty}=0$, $\int_{\mathbb{R}^{3}}\rho_{\alpha}(x){\rm d}x=1$ and
	\begin{equation}\label{ineq:binding-2}
	D(\rho_{\alpha},\rho_{\alpha}) \geq D(\rho_{0},\rho_{0}) + D(\tilde{\rho}_{\infty},\tilde{\rho}_{\infty}).
	\end{equation}
In addition, by the superadditivity of the function $\rho\mapsto j_{m}(\rho)$ and $\rho_{0}\tilde{\rho}_{\infty}=0$, we have
	\begin{equation}\label{eq:binding-3}
	j_{m}(\rho_{\alpha}) = j_{m}(\rho_{0})+j_{m}(\tilde{\rho}_{\infty}).
	\end{equation}
	Combining \eqref{ineq:binding-1}, \eqref{ineq:binding-2} and \eqref{eq:binding-3} together with the negativity of the external potential $V$ and the translation invariance of $\mathcal{E}^{\infty}_{\tau}$ we conclude that
	$$
	\frac{\delta}{3}+\mathcal{E}_{\tau}\left(\rho_{0}\right)+\mathcal{E}^{\infty}_{\tau}(\rho_{\infty})<E_{\tau}(1)\leq\mathcal{E}_{\tau}(\rho_{\alpha})\leq\mathcal{E}_{\tau}(\rho_{0})+\mathcal{E}^{\infty}_{\tau}({\rho}_{\infty}).
	$$		
	This is a contradiction and this implies that we must have \eqref{ineq:binding-0}. 
\end{proof}
\begin{lemma}[Coercivity of $\mathcal{E}_{\tau}$] \label{coercivity of E}
	Fix $q\geq 1$ and $m>0$. Assume that $V \not \equiv 0$ satisfies conditions $(V_1)$--$(V_2)$. Then for $0<\tau<\tau_{c}$, the energy functional $\mathcal{E}_{\tau}$ is coercive on $\left\{0\leq\rho\in L^{1}\cap L^{\frac{4}{3}}(\mathbb{R}^{3}),\int_{\mathbb{R}^{3}}\rho(x){\rm d}x=1\right\}$, i.e. we have
	$$
	\mathcal{E}_{\tau}(\rho)\to\infty \quad \text{as} \quad \int_{\mathbb{R}^{3}}\rho(x)^{\frac{4}{3}}{\rm d}x\to\infty.
	$$	
	In particular, all minimizing sequences of $\mathcal{E}_{\tau}$
	on $L^{1}\cap L^{\frac{4}{3}}(\mathbb{R}^{3})$ are bounded. 
\end{lemma}

\begin{proof}
	For any $\rho\in L^{1}\cap L^{\frac{4}{3}}(\mathbb{R}^{3})$ with $\int_{\mathbb{R}^{3}}\rho(x){\rm d}x=1$
	and $0<\epsilon<1$ we have 
	$$
	\mathcal{E}_{\tau}(\rho) = \epsilon\int_{\mathbb{R}^{3}}\rho(x)^{\frac{4}{3}}{\rm d}x + (1-\epsilon)\mathcal{E}_{\frac{\tau}{1-\epsilon}}(\rho) \geq \epsilon\int_{\mathbb{R}^{3}}\rho(x)^{\frac{4}{3}}{\rm d}x+(1-\epsilon)E_{\frac{\tau}{1-\epsilon}}(1).
	$$
	Since $0<\tau<\tau_{c}$ we can choose $\epsilon$ small enough such that
	$\frac{\tau}{1-\epsilon}<\tau_{c}$, and hence $E_{\frac{\tau}{1-\epsilon}}(1)>-\infty$.
	This implies that $\mathcal{E}_{\tau}(\rho)\to\infty$ as $\int_{\mathbb{R}^{3}}\rho(x)^{\frac{4}{3}}{\rm d}x\to\infty$. 
\end{proof}

By Lemma \ref{coercivity of E}, the minimizing sequence $\{\rho_{n}\}$ is bounded in $L^{\frac{4}{3}}(\mathbb{R}^{3})$, and hence there exists a subsequence $\{\rho_{n_k}\}$ such that $\rho_{n_k}\rightharpoonup \rho_{0}$ weakly in $L^{\frac{4}{3}}(\mathbb{R}^{3})$. We now apply the following adaptation of the concentration-compactness lemma.

\begin{lemma}\label{lem:concentration}
	Let $\{\rho_{n}\}_{n\geq 1}$ be a bounded sequence in $L^{\frac{4}{3}}(\mathbb{R}^{3})$ satisfying $\rho_{n} \geq 0$ and $\int_{\mathbb{R}^3}\rho_{n}(x){\rm d}x=1$. Then there exists a subsequence $\{\rho_{n_{k}}\}_{k\geq 1}$ satisfying one of the three following possibilities
	\begin{itemize}
		\item[(i)] (compactness) $\rho_{n_{k}}$ is tight, i.e. for all $\epsilon>0$, there exists $R<\infty$ such that
		$$
		\int_{|x|\leq R}\rho_{n_{k}}(x){\rm d}x\geq 1-\epsilon.
		$$
		\item[(ii)] (vanishing) ${\lim_{k\to\infty}\int_{|x|\leq R}\rho_{n_{k}}(x){\rm d}x=0}$ for all $R<\infty$. 
		\item[(iii)] (dichotomy) there exist $\alpha\in (0,1)$ and a sequence $\{R_{k}\}_{k\in\mathbb{N}}\subset\mathbb{R}_{+}$ with $R_{k}\to\infty$ such that
\begin{equation}\label{split}
		\lim_{k\to\infty}\int_{|x|\leq R_{k}}\rho_{n_{k}}(x){\rm d}x = \alpha, \quad \lim_{k\to\infty}\int_{R_{k}\leq |x| \leq 6R_{k}}\rho_{n_{k}}(x){\rm d}x=0.
\end{equation}		
	\end{itemize}
\end{lemma}

\begin{proof}[Sketch of the proof]
We will not detail the proof which is an adaptation of ideas by Lions \cite{Li-84}, where one introduces another sequence of concentration functions 
$$
f_{n}(t)=\int_{B(0,t)}\rho_{n}(x){\rm d}x.
$$
Then, by Helly's selection principle \cite{He-21}, there exist a subsequence $(n_{k})_{k\geq 1}$ and a nondecreasing non-negative function $f$ such that $f_{n_k}(t)\to f(t)$ as $k\to\infty$ for all $t>0$. Since $\rho_{n_k}\rightharpoonup \rho_{0}$ weakly in $L^{\frac{4}{3}}(\mathbb{R}^{3})$, the number $\alpha$ in (iii) is defined by
$$
\alpha=\lim_{t\to\infty}f(t)=\int_{\mathbb{R}^3}\rho_{0}(x){\rm d}x.
$$
We refer to \cite[Lemma 3.1]{Fr-03} for a similar proof of (iii).
\end{proof}
Invoking Lemma \ref{lem:concentration}, we obtain that a suitable subsequence $\{\rho_{n_{k}}\}$, with $\rho_{n_k}\rightharpoonup \rho_{0}$, satisfies either (i), (ii), or (iii). We now rule out (ii) and (iii) as follows.

\emph{Vanishing does not occur.} 
Suppose that $\{\rho_{n_{k}}\}$ exhibits property (ii), we deduce from it and the weak convergence $\rho_{n_{k}}\rightharpoonup \rho_{0}$ in $L^{\frac{4}{3}}(\mathbb{R}^{3})$ that $\int_{B(0,R)}\rho_{0}(x){\rm d}x=0$ for all $R<\infty$. This implies that $\rho_{0}=0$ almost everywhere in $\mathbb{R}^3$. Then we infer from the weak limit $\rho_{n_{k}}\rightharpoonup 0$ in $L^{\frac{4}{3}}(\mathbb{R}^{3})$ and the conditions $(V_1)$--$(V_2)$ that we must have
$$
\lim_{k\to\infty}\int_{\mathbb{R}^{3}}V(x)\rho_{n_{k}}(x){\rm d}x = 0.
$$
Thus, we obtain that
\begin{equation} \label{e-e}
E_{\tau}(1)=\lim_{k\to\infty}\mathcal{E}_{\tau}(\rho_{n_{k}})\geq E_{\tau}^{\infty}(1).
\end{equation}
It is well-known result (see \cite[Theorem 5]{LiYa-87}) that, up to translation, $E_{\tau}^{\infty}(1)$ possesses a unique minimizer $\rho_{\infty}$. By the negativity of $V$ we have
$$
E_{\tau}^{\infty}(1)=\mathcal{E}^{\infty}_{\tau}(\rho_{\infty})\geq E_{\tau}(1)-\int_{\mathbb{R}^3}V(x)\rho_{\infty}(x){\rm d}x > E_{\tau}(1),
$$
which contradicts \eqref{e-e}. Hence (ii) cannot occur.

\emph{Dichotomy does not occur.} 
Let us suppose that (iii) is true for $\{\rho_{n_{k}}\}$. Let $0\leq \chi \leq 1$ be a fixed smooth function on $\mathbb R^3$ such that $\chi(x) \equiv 1$ for $|x|<1$ and $\chi(x) \equiv 0$ for $|x|\geq 2$. Given the sequence $\{R_{k}\}$ from Lemma \ref{lem:concentration}, we define the functions $\chi_{R_{k}}(x) = \chi(x/R_{k})$ and $\zeta_{R_{k}}(x) = \sqrt{1-\chi_{R_{k}}(x)^{2}}$. Likewise, we define the sequences $\{\rho_{k}^{(1)}\}_{k\in\mathbb N}$ and $\{\rho_{k}^{(2)}\}_{k\in\mathbb N}$ by
$$
\rho_{k}^{(1)}(x) = \chi_{R_{k}}(x)^{2}\rho_{n_{k}}(x) \quad \text{and} \quad \rho_{k}^{(2)}(x) = \zeta_{R_{k}}(x)^{2}\rho_{n_{k}}(x).
$$
The direct term is separated as follows
\begin{equation}\label{non-dichotomy-1}
D(\rho_{n_{k}},\rho_{n_{k}}) =  D(\rho_{k}^{(1)},\rho_{k}^{(1)}) + D(\rho_{k}^{(2)},\rho_{k}^{(2)}) + 2D(\rho_{k}^{(1)},\rho_{k}^{(2)}).
\end{equation}
To show that the last term in \eqref{non-dichotomy-1} is of order one, we write 
$$
\zeta_{R_{k}}(y)^{2} = \zeta_{3R_{k}}(y)^{2} + \zeta_{R_{k}}(y)^{2} - \zeta_{3R_{k}}(y)^{2}
$$
and remark that $\chi_{R_{k}}(x)^{2}|x-y|^{-1}\zeta_{3R_{k}}(y)^{2} \leq R_{k}^{-1}$. So it remains to treat the term with $\chi_{R_{k}}(x)^{2}[\zeta_{R_{k}}(y)^{2}-\zeta_{3R_{k}}(y)^{2}]$, for which we use the Hardy--Littlewood--Sobolev inequality (see \cite[Theorem 4.3]{LiLo}) and \eqref{split} to obtain
$$
D(\chi_{R_{k}}^{2}\rho_{n_{k}},(\zeta_{R_{k}}^{2}-\zeta_{3R_{k}}^{2})\rho_{n_{k}}) \leq C\|\rho_{n_{k}}\|_{L^{\frac{4}{3}}}^{\frac{4}{3}} \|\rho_{n_{k}}\|_{L^{1}}^{\frac{1}{3}} \|\rho_{n_{k}}\mathbbm{1}(R_{k}\leq |\cdot| \leq 6R_{k})\|_{L^{1}}^{\frac{1}{3}} = o(1)_{k\to\infty}.
$$
On the other hand, since $V$ satisfies $(V_1)$--$(V_2)$, we have
\begin{equation}\label{non-dichotomy-2}
\int_{\mathbb{R}^{3}}V(x)\rho_{n_{k}}(x){\rm d}x = \int_{\mathbb{R}^{3}}V(x)\rho_{k}^{(1)}(x){\rm d}x + o(1)_{k\to\infty}.
\end{equation}
In addition, by the superadditivity of the function $\rho\mapsto j_{m}(\rho)$, we have
\begin{equation}\label{non-dichotomy-3}
\int_{\mathbb{R}^3}j_{m}(\rho_{n_{k}}){\rm d}x\geq \int_{\mathbb{R}^3}j_{m}(\rho_{k}^{(1)}){\rm d}x + \int_{\mathbb{R}^3}j_{m}(\rho_{k}^{(2)}){\rm d}x.
\end{equation}

Combining \eqref{non-dichotomy-1}, \eqref{non-dichotomy-2} and \eqref{non-dichotomy-3} we have
\begin{equation}\label{ineq:split energy}
\mathcal{E}_{\tau}(\rho_{n_{k}})\geq\mathcal{E}_{\tau}(\rho_{k}^{(1)})+\mathcal{E}^{\infty}_{\tau}(\rho_{k}^{(2)}) + o(1)_{k\to\infty}.
\end{equation}
We infer from \eqref{ineq:split energy} that
\begin{equation}\label{ineq:invert-binding-0}
E_{\tau}(1)=\lim_{k\to\infty}\mathcal{E}_{\tau}(\rho_{n_{k}})\geq E_{\tau}(\alpha)+E_{\tau}^{\infty}(1-\alpha),
\end{equation}
where we used that $\int_{\mathbb R^3}\rho_{k}^{(1)}(x){\rm d}x \to \alpha$, by \eqref{split}, and the continuity of $E_{\tau}(\alpha)$ and $E_{\tau}^{\infty}(\alpha)$ in $0<\alpha<1$. In summary, it follows from \eqref{ineq:invert-binding-0} and \eqref{ineq:binding-0} that 
\begin{equation}\label{eq:split energy}
E_{\tau}(1)= E_{\tau}(\alpha)+E_{\tau}^{\infty}(1-\alpha).
\end{equation}

Moreover, $\{\rho_{k}^{(1)}\}$ and $\{\rho_{k}^{(2)}\}$ are minimizing sequences of $E_{\tau}(\alpha)$ and $E_{\tau}^{\infty}(1-\alpha)$, respectively.
Note that, it follows from a simple scaling $\rho(x)\mapsto \rho((1-\alpha)^{-\frac{1}{3}}x)$
that $E_{\tau}^{\infty}(1-\alpha) = (1-\alpha)E_{(1-\alpha)^{\frac{2}{3}}\tau}^{\infty}(1)$. Since $(1-\alpha)^{\frac{2}{3}}\tau<\tau<\tau_{c}$, we deduce from \cite[Theorem 5]{LiYa-87} that $E_{\tau}^{\infty}(1-\alpha)$ has a unique minimizer, up to translations. Since we already know that $\int_{\mathbb R^3}\rho_{k}^{(1)}(x){\rm d}x \to \int_{\mathbb R^3}\rho_{0}(x){\rm d}x$, we claim by the same arguments in \cite[page 125]{Li-84} that we must have
$$
\lim_{k\to\infty}D(\rho_{k}^{(1)},\rho_{k}^{(1)}) = D(\rho_{0},\rho_{0}).
$$
On the other hand, it follows from the conditions $(V_1)$--$(V_2)$ that
$$
\lim_{k\to\infty}\int_{\mathbb{R}^{3}}V(x)\rho_{k}^{(1)}(x){\rm d}x = \int_{\mathbb{R}^{3}}V(x)\rho_{0}(x){\rm d}x.
$$
In addition, the convex functional $\int_{\mathbb{R}^{3}}j_{m}(\rho(x)){\rm d}x$ being strongly lower semi-continuous on $L^{\frac{4}{3}}(\mathbb R^3)$ by Fatou's lemma, it is weakly lower semi-continuous and we have
$$
\liminf_{k\to\infty}\int_{\mathbb{R}^{3}}j_{m}(\rho_{k}^{(1)}(x)){\rm d}x\geq\int_{\mathbb{R}^{3}}j_{m}(\rho_{0}(x)){\rm d}x.
$$
Hence, we conclude that
$$
E_{\tau}(\alpha) = \lim_{k\to\infty}\mathcal{E}_{\tau}(\rho_{k}^{(1)}) \geq \mathcal{E}_{\tau}(\rho_{0}) \geq E_{\tau}(\alpha).
$$

This implies that $\rho_{0}>0$ is a minimizer of $E_{\tau}(\alpha)$, and it satisfies the following variational equation
\begin{equation}\label{compact}
\sqrt{\eta_{0}(x)^{2}+m^{2}}-\tau(|\cdot|^{-1}\star\rho_{0})(x)+V(x)-\mu=0
\end{equation}
where $\eta_{0}=(6\pi^2\rho_{0}/q)^{\frac{1}{3}}$ and $\mu$ is a real-valued Lagrange multiplier. We note that \eqref{compact} implies that $\rho_{0}$ is compactly supported. If not, by letting $|x|\to\infty$ one has that $\mu\geq m$, since $(|\cdot|^{-1}\star\rho_{0})(x)$
and $V(x)$ tend to zero in \eqref{compact}. We then would
have by \eqref{compact}, $\rho_{0}(x)\geq C(|\cdot|^{-1}\star\rho_{0})^{3}(x)$, where constant $C$ is positive. For sufficiently large $|x|$, we see that $\rho_{0}(x)\geq C|x|^{-3}$. This implies that $\rho_{0}$ is not integrable, contradicting the
fact that $\int_{\mathbb{R}^{3}}\rho_{0}(x){\rm d}x=\alpha$. By the same argument we can also prove that the minimizer of $E_{\tau}^{\infty}(1-\alpha)$
is compactly supported (see also \cite[Appendix A]{LiOx-81}).

\begin{lemma}[Strict binding inequality]\label{ineq:strict-binding}
	Fix $q\geq 1$ and $m>0$. Assume that $V$ satisfies $(V_1)$--$(V_2)$. Then for $0<\alpha<1$ as above, we have
	$$
	E_{\tau}(1)<E_{\tau}(\alpha)+E_{\tau}^{\infty}(1-\alpha).
	$$
\end{lemma}

\begin{proof}
	We assume that $E_{\tau}^{\infty}(1-\alpha)$ possesses a minimizer $\rho_{\infty}$. As in the proof of Lemma~\ref{lem:ineq:binding}, we denote by $R>0$ the radius of a ball in $\mathbb{R}^{3}$ which contains the supports of $\rho_{0}$ and $\rho_{\infty}$, and we define the same translated operator $\tilde{\rho}_{\infty}(x)=\rho_{\infty}(x-3R)$ and the trial density operator $\rho_{\alpha}:=\rho_{0}+\tilde{\rho}_{\infty}$. By construction we have $\int_{\mathbb{R}^{3}}\rho_{\alpha}(x){\rm d}x=1$ and $\rho_{0}\tilde{\rho}_{\infty}=0$. Noticing that $\rho_{0}(x)\tilde{\rho}_{\infty}(y)=0$ when $|x-y|>5R$, we have
	\begin{equation}\label{ineq:strict-binding-1}
	-D(\rho_{\alpha},\rho_{\alpha}) + D(\rho_{0},\rho_{0}) + D(\tilde{\rho}_{\infty},\tilde{\rho}_{\infty}) = -2D(\rho_{0},\tilde{\rho}_{\infty})\leq -\frac{\alpha(1-\alpha)}{5R}<0.
	\end{equation}
	Combining \eqref{eq:binding-3} and \eqref{ineq:strict-binding-1} together with the negativity of the external potential $V$ and the translation invariance of $\mathcal{E}^{\infty}_{\tau}$ we conclude that
	$$
	E_{\tau}(1) \leq \mathcal{E}_{\tau}(\rho_{\alpha}) <\mathcal{E}_{\tau}(\rho_{0})+\mathcal{E}^{\infty}_{\tau}(\rho_{\infty}) = E_{\tau}(\alpha)+E_{\tau}^{\infty}(1-\alpha).
	$$
\end{proof}

The Lemma \ref{ineq:strict-binding}, together with \eqref{eq:split energy}, gives us a contradiction. Therefore (iii) of Lemma \ref{lem:concentration} is ruled out.

\emph{Conclusion of the proof of Theorem~\ref{thm:existence of minimizers} (iii).} 
We have shown that there exists a subsequence $\{\rho_{n_{k}}\}$ such that (i) of Lemma \ref{lem:concentration} holds true. Then we have
$$
1 \geq \int_{\mathbb{R}^3}\rho_{0}(x){\rm d}x \geq \int_{|x|\leq R}\rho_{0}(x){\rm d}x = \lim_{k\to\infty}\int_{|x|\leq R}\rho_{n_{k}}(x){\rm d}x \geq 1-\epsilon,
$$
for every $\epsilon>0$ and suitable $R = R(\epsilon) < \infty$. This implies that $\int_{\mathbb{R}^{3}}\rho_{0}(x){\rm d}x=1$. Now we prove that $\rho_{0}$ is indeed a minimizer of $E_{\tau}(1)$. We first deduce from the norm preservation and the same arguments in \cite[page 125]{Li-84} that we have 
\begin{equation}\label{ineq:strict-binding-min-1}
\lim_{k\to \infty}D(\rho_{n_k},\rho_{n_k})= D(\rho_{0},\rho_{0}).
\end{equation}
On the other hand, it follows from the conditions $(V_1)$--$(V_2)$ that
\begin{equation}\label{ineq:strict-binding-min-2}
\lim_{k\to\infty}\int_{\mathbb{R}^{3}}V(x)\rho_{n_k}(x){\rm d}x = \int_{\mathbb{R}^{3}}V(x)\rho_{0}(x){\rm d}x.
\end{equation}
In addition, the convex functional $\int_{\mathbb{R}^{3}}j_{m}(\rho(x)){\rm d}x$ being strongly lower semi-continuous on $L^{\frac{4}{3}}(\mathbb R^3)$ by Fatou's lemma, it is weakly lower semi-continuous and we have
\begin{equation}\label{ineq:strict-binding-min-3}
\liminf_{k\to\infty}\int_{\mathbb{R}^{3}}j_{m}(\rho_{n_k}(x)){\rm d}x\geq\int_{\mathbb{R}^{3}}j_{m}(\rho_{0}(x)){\rm d}x.
\end{equation}
Combining \eqref{ineq:strict-binding-min-1}, \eqref{ineq:strict-binding-min-2} and \eqref{ineq:strict-binding-min-3} we obtain
$$
E_{\tau}(1)=\lim_{k\to\infty}\mathcal{E}_{\tau}(\rho_{n_k})\geq\mathcal{E}_{\tau}(\rho_{0})\geq E_{\tau}(1)
$$
which implies that $\rho_{0}$ is a minimizer of $E_{\tau}(1)$.

\emph{Proof of Theorem~\ref{thm:existence of minimizers} (i)-(ii).}
To prove that there is no minimizer of \eqref{eq:neutron star energy} as soon as $\tau\geq \tau_c$ and $V \not \equiv 0$ satisfies conditions $(V_1)$--$(V_2)$, we proceed as follow. Let $Q$ be an optimizer in \eqref{ineq:HLS}. Since $\sqrt{|p|^2+m^2}\leq |p|+m^2/(2|p|)$, we find that $j_{m}(\rho)\leq K_{\rm cl}\rho^{\frac{4}{3}}+\frac{9}{16}m^2K_{\rm cl}^{-1}\rho^{\frac{2}{3}}$. Using this we have, for $\ell>0$ and $x_0\in\mathbb{R}^3$,
\begin{align}\label{non-existence}
\mathcal{E}_{\tau}(\ell^{3}Q(\ell(x-x_0))\leq&\left(1-\frac{\tau}{\tau_{c}}\right)\ell K_{\rm cl}\int_{\mathbb{R}^3} Q(x)^{\frac{4}{3}}{\rm d}x+\frac{9m^{2}}{16\ell K_{\rm cl}}\int_{\mathbb{R}^3} Q(x)^{\frac{2}{3}}{\rm d}x\nonumber \\
&+\int_{\mathbb{R}^3}V(\ell^{-1}x+x_0)Q(x){\rm d}x.
\end{align}
Since the function $x\mapsto Q(x)$ has compact support (see, e.g. \cite[Appendix A]{LiOx-81}), the convergence
$$
\lim_{\ell\to\infty}\int_{\mathbb{R}^3}V(\ell^{-1}x+x_0)Q(x){\rm d}x = V(x_0)
$$
holds for almost every $x_0\in\mathbb{R}^3$ (see, e.g.  \cite{LiLo}).

Hence, it follows from \eqref{non-existence} that, for $\tau=\tau_c$ and $V$ satisfies $(V_1)$--$(V_2)$, 
$$
E_{\tau}(1)\leq\lim_{\ell\to\infty}\mathcal{E}_{\tau}(\ell^{3}Q(\ell x))=\inf_{x\in\mathbb{R}^3}V(x).
$$
We argue that there does not exist a minimizer of $E_{\tau}(1)$ with $\tau=\tau_{c}$ by contradiction as follows. We suppose that $\rho\in L^{1}\cap L^{\frac{4}{3}}(\mathbb R^3)$ is a minimizer of $E_{\tau}(1)$ with $\tau=\tau_{c}$. It follows from the strict inequality $\sqrt{|p|^2+m^2}>|p|$ that
$$
\inf_{x\in\mathbb{R}^3}V(x)\geq \mathcal{E}_{\tau}(\rho) > {\mathcal{E}_{\tau}(\rho)}|_{m=0}\geq\inf_{x\in\mathbb{R}^3}V(x)
$$
which is a contradiction. Hence we have proved that, when $\tau=\tau_c$, no minimizer exists for
$E_{\tau}(1)=\inf_{x\in\mathbb{R}^3}V(x)$.

For $\tau>\tau_{c}$, it follows easily from \eqref{non-existence} that 
$$
E_{\tau}(1)\leq\lim_{\ell\to\infty}\mathcal{E}_{\tau}(\ell^{3}Q(\ell x))=-\infty.
$$
This implies that $E_{\tau}(1)$ is unbounded from below for any $\tau>\tau_{c}$, and hence the non-existence of minimizers of $E_{\tau}(1)$ is therefore proved.

\section{Blow-up behavior} 
\label{sec:behavior}

In this section, we prove the blow-up profile of minimizers of $E_{\tau}(1)$ when $\tau\uparrow \tau_{c}$, as stated in Theorem \ref{thm:behavior}. Let $\tau_{n}\uparrow \tau_{c}$ as $n\to\infty$ and let $\rho_{n}:=\rho_{\tau_{n}}$ be a non-negative minimizer of $E_{\tau_{n}}(1)$. We start with the following two preliminary lemmas.

\begin{lemma}
	There exist constants $M_{2}^\infty>M_{1}^\infty>0$ and $M_{1}>M_{2}>0$ independent of $\tau_{n}$ such that, for $n$ sufficiently large,
	\begin{equation}\label{ineq:estimate e^0}
	M_{1}^\infty(\tau_{c}-\tau_{n})^{\frac{1}{2}}\leq E^{\infty}_{\tau_{n}}(1)\leq M_{2}^\infty(\tau_{c}-\tau_{n})^{\frac{1}{2}}
	\end{equation} 
	if $V=0$, and
	\begin{equation}\label{ineq:estimate e^s}
	-M_{1}(\tau_{c}-\tau_{n})^{\frac{s}{s-1}}\leq E_{\tau_{n}}(1)\leq -M_{2}(\tau_{c}-\tau_{n})^{\frac{s}{s-1}}
	\end{equation} 
	if $V$ is defined as in \eqref{external potential}.
\end{lemma}

\begin{proof}
	We start with the proof of the upper bound in \eqref{ineq:estimate e^0} and \eqref{ineq:estimate e^s}. If $V$ is defined as in \eqref{external potential}, it follows from \eqref{non-existence} that, for $\ell>0$,
	$$
	E_{\tau_{n}}(1)\leq\left(1-\frac{\tau}{\tau_{c}}\right)\ell K_{\rm cl}\int_{\mathbb{R}^3} Q(x)^{\frac{4}{3}}{\rm d}x+\frac{9 m^{2}}{16\ell K_{\rm cl}}\int_{\mathbb{R}^3} Q(x)^{\frac{2}{3}}{\rm d}x-\ell^{s}\int_{\mathbb{R}^3}\frac{Q(x)}{|x|^{s}}{\rm d}x.
	$$
	By taking $\ell=C(\tau_{c}-\tau_{n})^{\frac{1}{s-1}}$, for some suitable positive constant $C$, we obtain the desired upper bound in \eqref{ineq:estimate e^s}. In the case $V=0$, the term $-\ell^{s}\int_{\mathbb{R}^3}\frac{Q(x)}{|x|^{s}}{\rm d}x$ does not appear in the above estimation, hence the desired upper bound in \eqref{ineq:estimate e^0} follows by taking $\ell=C(\tau_{c}-\tau_{n})^{-\frac{1}{2}}$.	
	
	Next we prove the lower bound in \eqref{ineq:estimate e^0}. From \eqref{ineq:HLS} and the upper bound of $E^{\infty}_{\tau_{n}}(1)$ in \eqref{ineq:estimate e^0} we see that
	$$
	M_{2}^\infty (\tau_{c}-\tau_{n})^{\frac{1}{2}} \geq \left(1-\frac{\tau_{n}}{\tau_{c}}\right)\int_{\mathbb{R}^3}j_{m}(\rho_{n}(x)){\rm d}x,
	$$
	which implies that
	\begin{equation}\label{fake-kinetic-0}
	\int_{\mathbb{R}^3}j_{m}(\rho_{n}(x)){\rm d}x\leq M_{2}^\infty \tau_{c} (\tau_{c}-\tau_{n})^{-\frac{1}{2}}.
	\end{equation}
	On the other hand, from the operator inequality $\sqrt{|p|^2+m^2} \geq |p| + m^2/(2\sqrt{|p|^2+m^2})$ we have
	\begin{equation}\label{ineq:fake-kinetic}
	\int_{\mathbb{R}^3}j_{m}(\rho_{n}(x)){\rm d}x\geq K_{\rm cl}\int_{\mathbb{R}^3}\rho_{n}(x)^{\frac{4}{3}}{\rm d}x + \frac{m^2}{2}\int_{\mathbb{R}^3}\tilde{j}_{m}(\rho_{n}(x)){\rm d}x
	\end{equation}
	where
	\begin{align}
	\tilde{j}_{m}(\rho)&=\frac{q}{(2\pi)^3}\int_{|p|<(6\pi^2\rho/q)^{\frac{1}{3}}}\frac{1}{\sqrt{|p|^2+m^2}}{\rm d}p \nonumber \\
	& = \frac{q}{4\pi^2}\left[\eta\sqrt{\eta^2+m^2}-m^2\ln{\left(\frac{\eta+\sqrt{\eta^2+m^2}}{m}\right)}\right], \quad \eta=\left(\frac{6\pi^2\rho}{q} \right)^{\frac{1}{3}}.\label{fake-kinetic}
	\end{align}
	Using H\"older's inequality and the fact that $j_{m}(\rho_{n}) \tilde{j}_{m}(\rho_{n})\geq \frac{9}{8}\rho_{n}$ we have
\begin{equation}\label{fake-kinetic-1}
	\int_{\mathbb{R}^3}j_{m}(\rho_{n}(x)){\rm d}x \int_{\mathbb{R}^3}\tilde{j}_{m}(\rho_{n}(x)){\rm d}x \geq \frac{9}{8}.
\end{equation}
	We deduce from \eqref{ineq:HLS}, \eqref{fake-kinetic-1} and \eqref{fake-kinetic-0} that
	$$
	E^{\infty}_{\tau_{n}}(1)  = \mathcal{E}_{\tau_{n}}^{\infty}(\rho_{n}) \geq \frac{m^{2}}{2}\int_{\mathbb{R}^3}\tilde{j}_{m}(\rho(x)){\rm d}x \geq  \frac{9m^{2}}{16\int_{\mathbb{R}^3}j_{m}(\rho(x)){\rm d}x} \geq \frac{9m^{2}}{16M_{2}^\infty \tau_{c} (\tau_{c}-\tau_{n})^{-\frac{1}{2}}}
	$$
	which is the lower bound in \eqref{ineq:estimate e^0} as desired.
	
	Now we come to prove the lower bound in \eqref{ineq:estimate e^s}, we proceed as follow. For every $1\leq i\leq M$ and for some $L>0$ small, it follows from H\"older's inequality that
	\begin{align}\label{ineq: bound V}
	\int_{\mathbb{R}^3}\frac{\rho_{n}(x)}{|x-x_i|^{s_{i}}}{\rm d}x &\leq \int_{|x-x_i|\leq L}\frac{\rho_{n}(x)}{|x-x_i|^{s_{i}}}{\rm d}x + \int_{|x-x_i|\geq L}\frac{\rho_{n}(x)}{|x-x_i|^{s_{i}}}{\rm d}x \nonumber \\ 
	&\leq L^{\frac{3-4s_{i}}{4}}\left(\int_{\mathbb{R}^3}\rho_{n}(x)^{\frac{4}{3}}{\rm d}x\right)^{\frac{3}{4}} + L^{-s_{i}}\nonumber \\
	&\leq L^{\frac{3-4s}{4}}\left(\int_{\mathbb{R}^3}\rho_{n}(x)^{\frac{4}{3}}{\rm d}x\right)^{\frac{3}{4}} + L^{-s},
	\end{align}
using $s=\max\{s_{i}:1\leq i\leq M\}$. We deduce from \eqref{ineq: bound V} and \eqref{ineq:HLS} that
	\begin{align}
		E_{\tau_{n}}(1) &\geq \left(1-\frac{\tau_{n}}{\tau_{c}}\right)\int_{\mathbb{R}^3}\rho_{n}(x)^{\frac{4}{3}}{\rm d}x - ML^{\frac{3-4s}{4}}\left(\int_{\mathbb{R}^3}\rho_{n}(x)^{\frac{4}{3}}{\rm d}x\right)^{\frac{3}{4}} - ML^{-s} \label{energy:lower-bound-V} \\
		& \geq -C\frac{L^{3-4s}}{(\tau_{c}-\tau_{n})^3} - ML^{-s}, \nonumber
	\end{align}
	where we have used Young's inequality for the first two terms on the right side of \eqref{energy:lower-bound-V}. Hence, the desired lower bound in \eqref{ineq:estimate e^s} follows by taking $L=(\tau_{c}-\tau_{n})^{\frac{1}{1-s}}$ for $n$ sufficiently large.
\end{proof}

\begin{lemma}
	There exist constants	$K_{2}^\infty>K_{1}^\infty>0$ and $K_{2}>K_{1}>0$ independent of $\tau_{n}$ such that, for $n$ sufficiently large,
	\begin{equation}\label{ineq:estimate D^0}
	K_{1}^\infty(\tau_{c}-\tau_{n})^{-\frac{1}{2}}\leq D(\rho_{n},\rho_{n})\leq 	K_{2}^\infty(\tau_{c}-\tau_{n})^{-\frac{1}{2}}
	\end{equation}
	if $V=0$, and 
	\begin{equation}\label{ineq:estimate D^s}
	K_{1}(\tau_{c}-\tau_{n})^{\frac{1}{s-1}}\leq D(\rho_{n},\rho_{n})\leq 	K_{2}(\tau_{c}-\tau_{n})^{\frac{1}{s-1}}
	\end{equation}
	if $V$ is defined as in \eqref{external potential}.

\end{lemma}
\begin{proof}
	We start with the proof of the upper bound in \eqref{ineq:estimate D^0} and \eqref{ineq:estimate D^s}. From \eqref{ineq:HLS} we see that it suffices to find an upper bound for $\int_{\mathbb{R}^3}\rho_{n}(x)^{\frac{4}{3}}{\rm d}x$. The upper bound in \eqref{ineq:estimate D^0} follows easily from \eqref{ineq:HLS} and the upper bound of $E^{\infty}_{\tau_{n}}(1)$ in \eqref{ineq:estimate e^0}. While the upper bound in \eqref{ineq:estimate D^s} follows from the upper bound of $E_{\tau_{n}}(1)$ in \eqref{ineq:estimate e^s} and \eqref{energy:lower-bound-V}, where we had chosen $L=M^{\frac{1}{s}}{M_2}^{-\frac{1}{s}}(\tau_{c}-\tau_{n})^{\frac{1}{1-s}}$ for $n$ sufficiently large. 
	
	Now let us only prove the lower bound in \eqref{ineq:estimate D^s}, since the proof of the lower bound in \eqref{ineq:estimate D^0} is analogous. For any $b$ such that $0\leq b\leq \tau_{n}$, we have
	\begin{equation}\label{ineq:direct term}
	E_{b}(1)\leq\mathcal{E}_{b}(\rho_{n})=E_{\tau_{n}}(1)+(\tau_{n}-b)D(\rho_{n},\rho_{n}).
	\end{equation}
	From \eqref{ineq:direct term} and \eqref{ineq:estimate e^s}, we deduce that there exist two positive constants
	$M_{1}>M_{2}$ such that for any $0<b<\tau_{n}<\tau_{c}$, 
	$$
	D(\rho_{n},\rho_{n})\geq\frac{E_{b}(1)-E_{\tau_{n}}(1)}{\tau_{n}-b}\geq\frac{-M_{1}\left(\tau_{c}-b\right)^{\frac{s}{s-1}}+M_{2}(\tau_{c}-\tau_{n})^{\frac{s}{s-1}}}{\tau_{n}-b}.
	$$
	Choosing $b=\tau_{n}-\beta(\tau_{c}-\tau_{n})$ with $\beta>0$, we arrive at 
	$$
	D(\rho_{n},\rho_{n}) \geq(\tau_{c}-\tau_{n})^{\frac{1}{s-1}}\frac{-M_{1}\left(1+\beta\right)^{\frac{s}{s-1}}+M_{2}}{\beta}.
	$$
	The last fraction is positive for $\beta$ large enough. Hence, for $\tau_{n}$ closes to $\tau_c$, there exists a positive constant $K_{1}$ such that
	$$
	D(\rho_{n},\rho_{n})\geq K_{1}(\tau_{c}-\tau_{n})^{\frac{s}{s-1}}.
	$$
\end{proof}

\begin{remark}
	 When $V$ is defined as in \eqref{external potential}, it follows from \eqref{ineq:estimate D^s} that $\int_{\mathbb{R}^3}\rho_{n}(x)^{\frac{4}{3}}{\rm d}x$ is large for $n$ sufficiently large. Hence, by taking $L^{-1} = \int_{\mathbb{R}^3}\rho_{n}(x)^{\frac{4}{3}}{\rm d}x$ in \eqref{ineq: bound V}, we obtain that there exists a constant $C>0$ such that 
	\begin{equation}\label{lower bound of V}
	\int_{\mathbb{R}^3}V(x)\rho_{n}(x){\rm d}x \geq -C\left(\int_{\mathbb{R}^3}\rho_{n}(x)^{\frac{4}{3}}{\rm d}x\right)^{s} \geq -C(\tau_{c}-\tau_{n})^{\frac{s}{s-1}},
	\end{equation}
	for $n$ sufficiently large.
\end{remark}

Now we are ready to complete the proof of Theorem \ref{thm:behavior}.

\begin{proof}[Proof of Theorem \ref{thm:behavior}] First, we focus on the case when $V$ is defined by \eqref{external potential}.

Let  $\epsilon_{n}:=(\tau_{c}-\tau_{n})^{\frac{1}{1-s}}>0$, we see that $\epsilon_{n}\to0$ as $n\to\infty$. For every $1\leq i\leq M$, we define the sequence of non-negative and $L^{1}(\mathbb R^3)$-normalized functions $w_{n}^{(i)}(x)=\epsilon_{n}^{3}\rho_{n}(\epsilon_{n}x+x_i)$. It follows from \eqref{ineq:HLS} and the upper bound of $E_{\tau_{n}}(1)$ in \eqref{ineq:estimate e^s} that there exists a positive constant $M_{2}$ such that
$$
\sum_{i=1}^{M}\int_{\mathbb{R}^3}\frac{z_{i}}{|x-x_i|^{s_{i}}}\rho_{n}(x){\rm d}x = -\int_{\mathbb{R}^3}V(x)\rho_{n}(x){\rm d}x \geq M_{2}\epsilon_{n}^{-s}.
$$
From this, we deduce that there exists $j$ verifying $1\leq j \leq M$ such that
\begin{equation}\label{kill s_{i}}
\epsilon_{n}^{s-s_{j}}\int_{\mathbb{R}^3}\frac{z_j}{|x|^{s_j}}w_{n}^{(j)}(x){\rm d}x = \epsilon_{n}^{s}\int_{\mathbb{R}^3}\frac{z_j}{|x-x_j|^{s_j}}\rho_{n}(x){\rm d}x \geq \frac{M_{2}}{M},
\end{equation}
which implies that $s_{j}=\max\{s_{i}:1\leq i\leq M\}=:s$. Otherwise, if $s_{j}<s$ then $\epsilon_{n}^{s-s_{j}}\to 0$ as $n\to\infty$, which contradicts \eqref{kill s_{i}}. Now, for such $j$, we deduce from \eqref{ineq:estimate D^s} that
\begin{equation}
D(w_{n}^{(j)},w_{n}^{(j)})>0,\label{ineq:vanishing1}
\end{equation}
and there exists a constant $C>0$ such that
$$\int_{\mathbb{R}^3}w_{n}^{(j)}(x)^{\frac{4}{3}}{\rm d}x= \epsilon_{n}\int_{\mathbb{R}^3}\rho_{n}(x)^{\frac{4}{3}}{\rm d}x\leq C.   
$$

Thus $\{w_{n}^{(j)}\}$ is bounded in $L^{\frac{4}{3}}(\mathbb{R}^{3})$, and hence there exists a subsequence of $\{w_{n}^{(j)}\}$ (still denote by $\{w_{n}^{(j)}\}$) such that $w_{n}^{(j)}\rightharpoonup w$ weakly in $L^{\frac{4}{3}}(\mathbb{R}^{3})$. Since $\rho_{n}$ is a non-negative minimizer of $E_{\tau_{n}}(1)$, it satisfies the following variational equations
\begin{equation}\label{eq:E-L-1}
\sqrt{\eta_{n}(x)^{2}+m^{2}}-\tau_{n}(|\cdot|^{-1}\star\rho_{n})(x)+V(x)-\mu_{n}\begin{cases}
=0 & \text{if }\rho_{n}(x)>0,\\
\geq0 & \text{if }\rho_{n}(x)=0.
\end{cases}
\end{equation} 
where $\eta_{n}=(6\pi^2\rho_{n}/q)^{\frac{1}{3}}$ and $\mu_{n}$ is Lagrange multiplier. In fact,
\begin{equation} \label{eq:mu_{n}}
\mu_{n}=\int_{\mathbb{R}^3}\sqrt{\eta_{n}(x)^{2}+m^2}\rho_{n}(x){\rm d}x-2\tau_{n}D(\rho_{n},\rho_{n})+\int_{\mathbb{R}^3}V(x)\rho_{n}(x){\rm d}x.
\end{equation}
We see that $w_{n}^{(j)}$ is a non-negative solution to
\begin{equation} \label{eq:E-L-2}
\sqrt{\zeta_{n}^{j}(x)^{2}+m^{2}\epsilon_n^2}-\tau_{n}(|\cdot|^{-1}\star w_{n}^{(j)})(x)+\epsilon_{n}V(\epsilon_{n}x+x_{j})-\epsilon_{n}\mu_{n}\begin{cases}
=0 & \text{if } w_{n}^{(j)}(x)>0,\\
\geq0 & \text{if } w_{n}^{(j)}(x)=0.
\end{cases}
\end{equation} 
where $\zeta_{n}^{j}=(6\pi^2 w_{n}^{(j)}/q)^{\frac{1}{3}}$. From the fact that
\begin{equation} \label{ineq:mu_{n}} j_m(\rho_{n}(x))\leq\sqrt{\eta_{n}(x)^{2}+m^2}\rho_{n}(x)\leq\frac{4}{3}j_m(\rho_{n}(x)),
\end{equation}
we have
$$
E_{\tau_{n}}(1)-\tau_{n}D(\rho_{n},\rho_{n})\leq \mu_{n}\leq \frac{4}{3}E_{\tau_{n}}(1)-\frac{2}{3}\tau_{n}D(\rho_{n},\rho_{n})-\frac{1}{3}\int_{\mathbb{R}^3}V(x)\rho_{n}(x){\rm d}x.
$$

Hence we deduce from \eqref{ineq:estimate e^s}, \eqref{ineq:estimate D^s} and \eqref{lower bound of V} that $\epsilon_{n}\mu_{n}$ is bounded
uniformly and strictly negative as $n\to\infty$. By passing to a subsequence if necessary, we can thus
assume that $\epsilon_{n}\mu_{n}$ converges to some number $-\alpha<0$ as $n\to\infty$. Passing \eqref{eq:E-L-2} to weak limit, we obtain that $w$ is a non-negative solution to
$$
\frac{4}{3}K_{\rm cl} w(x)^{\frac{1}{3}}-\tau_{c}(|\cdot|^{-1}\star w)(x)+\alpha\begin{cases}
=0 & \text{if } w(x)>0,\\
\geq0 & \text{if } w(x)=0.
\end{cases}
$$
By a simple rescaling we see that, $Q(x)=\lambda^{-3}w(\lambda^{-1}x+y_{0})$ is a non-negative solution of \eqref{eq:eq:LE} for $\lambda=\frac{3\alpha}{2\tau_{c}}$ and $y_{0}\in\mathbb{R}^3$. Now we claim that there exists a positive constant $R_0$ such that 
\begin{equation}
\liminf_{n\to\infty}\int_{B(0,R_0)}w_{n}^{(j)}(x){\rm d}x>0.\label{eq:vanishing2}
\end{equation}
On the contrary, we assume that for any $R>0$ there exists a subsequence of $\{\tau_{n}\}$ (still denoted by $\{\tau_{n}\}$), such that 
$$
\lim_{n\to\infty}\int_{B(0,R)}w_{n}^{(j)}(x){\rm d}x=0.
$$
Then by the same arguments in \cite[page 124]{Li-84} we can prove that $D(w_{n}^{(j)},w_{n}^{(j)})\to0$ as $n\to\infty$.
This contradicts \eqref{ineq:vanishing1}, hence \eqref{eq:vanishing2} holds true. It follows from \eqref{eq:vanishing2} and the weak limit $w_{n}^{(j)}\rightharpoonup w$ in $L^{\frac{4}{3}}(\mathbb{R}^{3})$ that
$$
\int_{B(0,R_0)}w(x){\rm d}x = \lim_{n\to\infty}\int_{B(0,R_0)}w_{n}^{(j)}(x){\rm d}x>0
$$
which implies that $w>0$ in $\mathbb{R}^3$. Hence $Q(x)>0$ in $\mathbb{R}^3$, and it solves the equation
$$
\frac{4}{3}\sigma_{f} Q(x)^{\frac{1}{3}}-(|\cdot|^{-1}\star Q)(x)+\frac{2}{3} = 0,
$$
which implies that
\begin{equation} \label{eq:Q}
\frac{2}{3}\sigma_{f} \|Q\|_{L^{\frac{4}{3}}}^{\frac{4}{3}}-D(Q,Q)+\frac{1}{3}\|Q\|_{L^1} = 0.
\end{equation}

We now prove that $Q$ is indeed an optimizer in \eqref{ineq:HLS}. Let $G$ be an optimizer in \eqref{ineq:HLS} with $\int_{\mathbb R^3} G(x){\rm d}x = 1$ and let $g(x)=\epsilon_{n}^{-3}G(\epsilon_{n}^{-1}x)$, then we have
\begin{align}
\epsilon_{n}\mathcal{E}_{\tau_{n}}(g) & \leq \left(1-\frac{\tau_{n}}{\tau_{c}}\right)K_{\rm cl}\int_{\mathbb{R}^3} G(x)^{\frac{4}{3}}{\rm d}x+\frac{9 m^{2}\epsilon^{2}_{n}}{16K_{\rm cl}}\int_{\mathbb{R}^3}G(x)^{\frac{2}{3}}{\rm d}x \nonumber \\
& \quad + \epsilon_{n}\int_{\mathbb{R}^3}V(\epsilon_{n}x)G(x){\rm d}x.\label{ineq:GN-min-1}
\end{align}
On the other hand, since $\mu_{n}$ satisfies \eqref{eq:E-L-2}, we deduce from \eqref{ineq:mu_{n}} that
\begin{align}
	\epsilon_{n}\mathcal{E}_{\tau_{n}}(\rho_{n}) & = \int_{\mathbb{R}^{3}}j_{m\epsilon_{n}}(w_{n}^{(j)}(x)){\rm d}x -\tau_{n} D(w_{n}^{(j)},w_{n}^{(j)}) + \epsilon_{n}\int_{\mathbb{R}^3} V(\epsilon_{n}x+x_{j})w_{n}^{(j)}(x){\rm d}x \nonumber \\
	&\geq \frac{3}{4}\int_{\mathbb{R}^{3}}w_{n}^{(j)}\sqrt{\zeta_{n}^2+m^2\epsilon_{n}^2}-\tau_{n} D(w_{n}^{(j)},w_{n}^{(j)}) + \epsilon_{n}\int_{\mathbb{R}^3} V(\epsilon_{n}x+x_{j})w_{n}^{(j)}(x){\rm d}x \nonumber\\
	& \geq \frac {3}{4}\epsilon_{n}\mu_{n} + \frac{1}{2}\tau_{n}D(w_{n}^{(j)},w_{n}^{(j)}) + \frac{1}{4}\epsilon_{n}\int_{\mathbb{R}^3} V(\epsilon_{n}x+x_{j})w_{n}^{(j)}(x){\rm d}x.\label{ineq:GN-min-2}
\end{align}	
Since $\rho_{n}$ is a minimizer of $E_{\tau_{n}}(1)$, we have $\mathcal{E}_{\tau_{n}}(\rho_{n})\leq\mathcal{E}_{\tau_{n}}(g)$ and hence
\begin{equation}\label{ineq:GN-min-3}
\liminf_{n\to \infty}\epsilon_{n}\mathcal{E}_{\tau_{n}}(\rho_{n})\leq\liminf_{n\to \infty}\epsilon_{n}\mathcal{E}_{\tau_{n}}(g).
\end{equation}
From \eqref{ineq:GN-min-1}, \eqref{ineq:GN-min-2}, \eqref{ineq:GN-min-3} and the fact that $D(\rho,\rho)$ is weakly lower semi-continuous, we deduce that 
$$
D(Q,Q)=\frac{1}{\lambda}D(\omega,\omega)\leq \frac{2}{3\alpha}\liminf_{n\to \infty}\tau_{n}D(w_{n}^{(j)},w_{n}^{(j)})\leq 1.
$$
Moreover, from \eqref{ineq:HLS}, \eqref{eq:Q} and the fact that $\|Q\|_{L^1}\leq 1$, we have 
$$
	D(Q,Q) = \frac{2}{3}\sigma_{f} \|Q\|_{L^{\frac{4}{3}}}^{\frac{4}{3}}+\frac{1}{3}\|Q\|_{L^1} \geq \left(\sigma_{f} \|Q\|_{L^{\frac{4}{3}}}^{\frac{8}{3}}\|Q\|_{L^1}\right)^{\frac{1}{3}}\geq  D(Q,Q)^{\frac{2}{3}}.
$$
This implies that $D(Q,Q)\geq 1$, and hence $D(Q,Q)=1$. From this, it is easy to see that $\sigma_{f} \|Q\|_{L^{\frac{4}{3}}}^{\frac{4}{3}}=1=\|Q\|_{L^1}$. Thus $Q$ is indeed an optimizer in \eqref{ineq:HLS}. Denote by $Q^*$ the symmetric-decreasing rearrangement of $Q$, then we have $\|Q\|_{L^p}=\|Q^*\|_{L^p}$ for all $1\leq p\leq \infty$. Hence, it follows from \eqref{ineq:HLS} and the Riesz's rearrangement inequality (see \cite[Theorem 3.7]{LiLo}) that
$$
1=\sigma_{f}\|Q^*\|_{L^1}^{\frac{2}{3}} \|Q^*\|_{L^{\frac{4}{3}}}^{\frac{4}{3}}\geq D(Q^*,Q^*)\geq D(Q,Q)=1.
$$
The equality occurs only if $Q(x)=Q^*(x-y)$ for some $y\in\mathbb{R}^3$ (see \cite[Theorem 3.9]{LiLo}). Thus, up to translation, $Q$ is a positive radially symmetric decreasing function which satisfies \eqref{cond:LE} and \eqref{eq:eq:LE}.

We shall show that $w_{n}^{(j)}\to w$ strongly in $L^{1} \cap L^{\frac{4}{3}}(\mathbb{R}^{3})$. We note that $\|w\|_{L^{1}}=\|Q\|_{L^{1}}=1$. From this norm preservation we have $\int_{\mathbb R^3}w_{n}^{(j)}(x){\rm d}x \to \int_{\mathbb R^3}w(x){\rm d}x$. Thus, we claim by the same arguments in \cite[page 125]{Li-84} that
$$
\lim_{n\to\infty}D(w_{n}^{(j)},w_{n}^{(j)})=D(w,w).
$$
We deduce from the above convergence and the inequality
$$
\epsilon_{n}\mathcal{E}_{\tau_{n}}(\epsilon_{n}^{-3}w_{n}^{(j)}(\epsilon_{n}^{-1}(x-x_{j}))) = \epsilon_{n}\mathcal{E}_{\tau_{n}}(\rho_{n}) \leq \epsilon_{n}\mathcal{E}_{\tau_{n}}(\epsilon_{n}^{-3}w(\epsilon_{n}^{-1}x))
$$
that
\begin{align*}
\limsup_{n\to\infty}K_{\rm cl}\int_{\mathbb{R}^3}w_{n}^{(j)}(x)^{\frac{4}{3}}{\rm d}x &\leq \limsup_{n\to\infty}\int_{\mathbb{R}^3}j_{m\epsilon_{n}}(w_{n}^{(j)}(x)){\rm d}x\\
& \leq \limsup_{n\to\infty}\int_{\mathbb{R}^3}j_{m\epsilon_{n}}(w(x)){\rm d}x = K_{\rm cl}\int_{\mathbb{R}^3}w(x)^{\frac{4}{3}}{\rm d}x.
\end{align*}
Since $w_{n}^{(j)}\rightharpoonup w$ weakly in $L^{\frac{4}{3}}(\mathbb{R}^{3})$, by Fatou's Lemma we have
$$
\liminf_{n\to\infty}\int_{\mathbb{R}^3}w_{n}^{(j)}(x)^{\frac{4}{3}}{\rm d}x\geq \int_{\mathbb{R}^3}w(x)^{\frac{4}{3}}{\rm d}x.
$$
Therefore we have proved that
\begin{equation}\label{conv:measure}
\lim_{n\to\infty}\int_{\mathbb{R}^3}w_{n}^{(j)}(x)^{\frac{4}{3}}{\rm d}x=\int_{\mathbb{R}^3}w(x)^{\frac{4}{3}}{\rm d}x.
\end{equation}
Since $\rho \mapsto \rho^{\frac{4}{3}}$ is strictly convex, we deduce from \eqref{conv:measure} that $w_{n}^{(j)}\to w$ in measure (see, e.g. \cite[page 126-127]{Li-84}). Thus, up to a subsequence, we have $w_{n}^{(j)}\to w$ pointwise almost everywhere in $\mathbb{R}^3$. Using this pointwise convergence, we deduce from the Brezis--Lieb refinement of Fatou's lemma (see, e.g. \cite[Theorem 1.9]{LiLo}) that
$$
\int_{\mathbb{R}^3}w_{n}^{(j)}(x)^{r}{\rm d}x = \int_{\mathbb{R}^3}w(x)^{r}{\rm d}x + \int_{\mathbb{R}^3}|w_{n}^{(j)}(x)-w(x)|^{r}{\rm d}x + o(1)_{n\to\infty}
$$
for $r=1$ or $r=\frac{4}{3}$. Therefore $\int_{\mathbb{R}^3}w_{n}^{(j)}(x)^{r}{\rm d}x \to \int_{\mathbb{R}^3}w(x)^{r}{\rm d}x$ implies that $w_{n}^{(j)}\to w$ strongly in $L^{1}\cap L^{\frac{4}{3}}(\mathbb{R}^3)$.

We have thus shown that there exists a subsequence of $\{\tau_{n}\}$ (still denoted by $\{\tau_{n}\}$) such that we have the following strong convergence in $L^{1}\cap L^{\frac{4}{3}}(\mathbb{R}^{3})$
$$
\lim_{n\to\infty}\epsilon_{n}^{3}\rho_{n}(\epsilon_{n} x+x_{j}) = \lim_{n\to\infty}w_{n}^{(j)}(x) = w(x) = \lambda^{3}Q(\lambda (x-y_{0})),
$$
where $\lambda>0$, $y_{0}\in\mathbb{R}^3$ and $Q$ is positive radially symmetric decreasing and optimizes the inequality \eqref{ineq:HLS}. To complete the proof of Theorem \ref{thm:behavior} (ii), we now determine the exact values of $\lambda$ and $y_{0}$. Since $\rho_{n}(x)=\epsilon_{n}^{-3}w_{n}^{(j)}(\epsilon_{n}^{-1}(x-x_{j}))$ is a minimizer of $E_{\tau_{n}}(1)$ we have, by using \eqref{ineq:fake-kinetic} and \eqref{ineq:HLS}, that
\begin{align}
	E_{\tau_{n}}(1) & \geq \int_{\mathbb{R}^3}\tilde{j}_{m}(\rho_{n}(x)){\rm d}x + \epsilon_{n}^{1-s}D(\rho_{n},\rho_{n})+\int_{\mathbb{R}^3}V(x)\rho_{n}(x){\rm d}x \nonumber \\
	& \geq \epsilon_{n}\int_{\mathbb{R}^3}\tilde{j}_{m\epsilon_{n}}(w_{n}^{(j)}(x)){\rm d}x+ \epsilon_{n}^{-s}D(w_{n}^{(j)},w_{n}^{(j)})+\int_{\mathbb{R}^3}V(\epsilon_{n}x+x_{j})w_{n}^{(j)}(x){\rm d}x.\label{liminf:E}
\end{align}
From the identity \eqref{fake-kinetic} for $\tilde{j}_{m\epsilon_{n}}(w_{n}^{(j)}(x))$ and Fatou's Lemma, we have 
\begin{align}\label{liminf:kinetic}
\liminf_{n\to\infty}\int_{\mathbb{R}^3}\tilde{j}_{m\epsilon_{n}}(w_{n}^{(j)}(x)){\rm d}x\geq\frac{q}{4\pi^2}\int_{\mathbb{R}^3}\tilde{\eta}(x)^{2}{\rm d}x =\frac{9}{8K_{\rm cl}\lambda}\int_{\mathbb{R}^3}Q(x)^{\frac{2}{3}}{\rm d}x,
\end{align}
where $\tilde{\eta}=(6\pi^2 w/q)^{\frac{1}{3}}$. By the Hardy--Littlewood--Sobolev inequality (see \cite[Theorem 4.3]{LiLo}), we have
\begin{align}\label{liminf:D}
\lim_{n\to\infty}D(w_{n}^{(j)},w_{n}^{(j)})=D(w,w)=\lambda D(Q,Q)=\lambda.
\end{align}
On the other hand, we have
\begin{align}\label{liminf:V}
\lim_{n\to\infty}\epsilon_{n}^{s}\int_{\mathbb{R}^3}V(\epsilon_{n}x+x_{j})w_{n}^{(j)}(x){\rm d}x =-z_{j}\int_{\mathbb{R}^3}\frac{w(x)}{|x|^s}{\rm d}x & = -\lambda^{s}z_{j}\int_{\mathbb{R}^3}\frac{Q(x+y_0)}{|x|^{s}}{\rm d}x \nonumber \\
& \geq -\lambda^{s}z\int_{\mathbb{R}^3}\frac{Q(x)}{|x|^{s}}{\rm d}x
\end{align}
since $Q$ is a radial decreasing function and $z=\max\{z_{i}:s_{i}=s\}$. It follows from \eqref{liminf:E}, \eqref{liminf:D} and \eqref{liminf:V} that
$$
\liminf_{n\to\infty}\frac{E_{\tau_{n}}(1)}{\epsilon_{n}^{-s}}\geq \lambda -\lambda^{s} z\int_{\mathbb{R}^3}\frac{Q(x)}{|x|^{s}}{\rm d}x.
$$
Thus, taking the infimum over $\lambda>0$ we obtain
\begin{equation}\label{liminf: e/epsilon}
\liminf_{n\to\infty}\frac{E_{\tau_{n}}(1)}{\epsilon_{n}^{-s}} \geq \left(1-\frac{1}{s}\right) \left(sz\int_{\mathbb{R}^3}\frac{Q(x)}{|x|^{s}}{\rm d}x\right)^{\frac{1}{1-s}}.
\end{equation}
To see the matching upper bound in \eqref{liminf: e/epsilon}, one simply takes 
$$
\rho_{n}(x)=(\tilde{\lambda}\epsilon_{n}^{-1})^{3}Q(\tilde{\lambda}\epsilon_{n}^{-1}(x-x_i))
$$
as trial state for $\mathcal{E}_{\tau_{n}}$, where $\tilde{\lambda}>0$ and $x_i\in\mathcal{Z}$, i.e. $s_{i}=s$ and $z_{i}=z$. We use \eqref{ineq:HLS} and the fact that $j_{m}(\rho_{n})\leq K_{\rm cl}\rho_{n}^{\frac{4}{3}}+\frac{9}{16}m^2K_{\rm cl}^{-1}\rho_{n}^{\frac{2}{3}}$ we obtain
\begin{align*}
	E_{\tau_{n}}(1)
	&\leq \frac{9m^2}{16K_{\rm cl}}\int_{\mathbb{R}^3}\rho_{n}(x)^{\frac{2}{3}}{\rm d}x + (\tau_c - \tau_{n})D(\rho_{n},\rho_{n}) + \int_{\mathbb{R}^3}V(x)\rho_{n}(x){\rm d}x\\
	&=\frac{9m^2\epsilon_{n}}{16K_{\rm cl}\tilde{\lambda}}\int_{\mathbb{R}^3}Q(x)^{\frac{2}{3}}{\rm d}x + \tilde{\lambda}\epsilon_{n}^{-s} + \int_{\mathbb{R}^3}V(\epsilon_{n}\tilde{\lambda}^{-1}x+x_i)Q(x){\rm d}x.
\end{align*}
This implies that
$$
\limsup_{n\to\infty}\frac{E_{\tau_{n}}(1)}{\epsilon_{n}^{-s}}\leq \tilde{\lambda}-\tilde{\lambda}^{s}z\int_{\mathbb{R}^3}\frac{Q(x)}{|x|^{s}}{\rm d}x.
$$
Thus, taking the infimum over $\tilde{\lambda}>0$ we see that 
\begin{equation}\label{limsup: e/epsilon}
\limsup_{n\to\infty}\frac{E_{\tau_{n}}(1)}{\epsilon_{n}^{-s}}\leq \left(1-\frac{1}{s}\right) \left(sz\int_{\mathbb{R}^3}\frac{Q(x)}{|x|^{s}}{\rm d}x\right)^{\frac{1}{1-s}}.
\end{equation}

From \eqref{liminf: e/epsilon} and \eqref{limsup: e/epsilon} we conclude that $z_{j} = \max\{z_{i}:s_{i}=s\} =: z$,
$$
\lambda = \left(sz\int_{\mathbb{R}^3}\frac{Q(x)}{|x|^{s}}{\rm d}x\right)^{\frac{1}{1-s}} \quad \text{and} \quad \lim_{n\to\infty}\frac{E_{\tau_{n}}(1)}{\epsilon_{n}^{-s}} = \left(1-\frac{1}{s}\right) \left(sz\int_{\mathbb{R}^3}\frac{Q(x)}{|x|^{s}}{\rm d}x\right)^{\frac{1}{1-s}}.
$$
We note that the limit of $E_{\tau_{n}}(1)/\epsilon_{n}^{-s}$ is independent of the subsequence $\tau_{n}$. Therefore, we have the convergence of the whole family of $\tau_{n}$ in \eqref{lim:e(tau)/epsilon^s}. Finally, since the limit in \eqref{conv:minimizer-V} is unique, if $\mathcal{Z}$ has a unique element, then we obtain the convergence \eqref{conv:minimizer-V} for the whole sequence of $\tau_{n}$ by a standard argument.

Now we come to the case when $V\equiv 0$. In this case, we define $\tilde{w}_{n}(x)=\epsilon_{n}^{3}\rho_{n}(\epsilon_{n}x)$ where $\epsilon_{n}:=(\tau_{c}-\tau_{n})^{\frac{1}{2}}$. It follows from \eqref{ineq:estimate D^0} that $D(\tilde{w}_{n},\tilde{w}_{n})>0$. By the same arguments as in \cite[page 124]{Li-84} we can prove that there exists a sequence $\{x_n\}\subset\mathbb{R}^3$ and a positive constant $R_1$ such that 
$$
\liminf_{n\to\infty}\int_{B(x_n,R_1)}\tilde{w}_{n}(x){\rm d}x>0.
$$
By the same arguments as we have done before for the case $V\not\equiv 0$ , we can prove that there exists a subsequence of $\{\tau_{n}\}$ (still denoted by $\{\tau_{n}\}$) and a sequence $\{x_n\}\subset\mathbb{R}^3$ such that we have the following strong convergence in $L^{1}\cap L^{\frac{4}{3}}(\mathbb{R}^{3})$
$$
\lim_{n\to\infty}\epsilon_{n}^{3}\rho_{n}(\epsilon_{n}x+\epsilon_{n}x_{n}) = \lim_{n\to\infty}\tilde{w}_{n}(x+x_n) = w(x) = \lambda_{\infty}^{3}Q(\lambda_{\infty}x),
$$
where $\lambda_{\infty}>0$ and $Q$ is positive radially symmetric decreasing and optimizes the inequality \eqref{ineq:HLS}. It remains to compute the exact value of $\lambda_{\infty}$, which is the consequence of the computation $\lim_{n\to\infty} E^{\infty}_{\tau_{n}}(1)/\epsilon_{n}$. The lower bound follows from \eqref{liminf:E}, \eqref{liminf:kinetic} and \eqref{liminf:D}, while the upper bound is done by taking the trial state
$$
\rho_{n}(x)=(\tilde{\lambda}\epsilon_{n}^{-1})^{3}Q(\tilde{\lambda}\epsilon_{n}^{-1}x),
$$
and by optimizing the quantity $\limsup_{n\to\infty} E^{\infty}_{\tau_{n}}(1)/\epsilon_{n}$ over all $\tilde{\lambda}$. Here $\tilde{\lambda}>0$ and $Q$ is positive radially symmetric decreasing and optimizes the inequality \eqref{ineq:HLS}. In summary, we conclude that 
$$
\lambda_{\infty} = \frac{3}{4}m\left(\frac{1}{K_{\rm cl}}\int_{\mathbb{R}^3}Q(x)^{\frac{2}{3}}{\rm d}x\right)^{\frac{1}{2}} \quad \text{and} \quad 
\lim_{n\to\infty} \frac{E^{\infty}_{\tau_{n}}(1)}{\epsilon_{n}} = \frac{3}{2}m\left(\frac{1}{K_{\rm cl}}\int_{\mathbb{R}^3}Q(x)^{\frac{2}{3}}{\rm d}x\right)^{\frac{1}{2}}.
$$
Since the limit of $E^{\infty}_{\tau_{n}}(1)/\epsilon_{n}$ is independent of the subsequence of $\tau_{n}$, we have the convergence of the whole family of $\tau_{n}$ in \eqref{lim:e(tau)/epsilon^0}. The proof is complete.
\end{proof}

\medskip
{\bf Acknowledgement.}
The author would like to thank Phan Th\`{a}nh Nam for helpful discussions. He also thanks the referee for helpful comments on the manuscript.

\end{document}